\documentclass{amsart}
\usepackage{verbatim}
\usepackage{amssymb}
\usepackage{amsmath}
\usepackage[usenames]{color}
\usepackage{graphicx}
\usepackage{amscd}

\theoremstyle{plain}
\newtheorem{theorem}{Theorem}

\newtheorem{proposition}[theorem]{Proposition}
\newtheorem{lemma}[theorem]{Lemma}
\newtheorem{corollary}[theorem]{Corollary}

\newtheorem{definition}[theorem]{Definition}

\theoremstyle{definition}

\newtheorem{remark}[theorem]{Remark}

\newtheorem{example}[theorem]{Example}
\numberwithin{equation}{section}
\numberwithin{theorem}{section}
\allowdisplaybreaks

\usepackage{hyperref}

\newcommand{\calL}{{\mathcal L}}

\newcommand{\one}{{{\bf 1}}}

\newcommand{\limn}{\lim_{n\to\infty}}

\newcommand{\ud}[0]{\,\mathrm{d}}

\begin{document}

\title[Brownian representations]
{Brownian representations of cylindrical continuous local martingales}

\author{Ivan Yaroslavtsev}
\address{Delft Institute of Applied Mathematics\\
Delft University of Technology \\ P.O. Box 5031\\ 2600 GA Delft\\The
Netherlands}
\email{I.S.Yaroslavtsev@tudelft.nl}

\begin{abstract}
In this paper we give necessary and sufficient conditions for a cylindrical continuous local martingale to be the stochastic integral with respect to a cylindrical Brownian motion. In particular we consider the class of cylindrical martingales with closed operator-generated covariations. We also prove that for every cylindrical continuous local martingale $M$ there exists a time change $\tau$ such that $M\circ \tau$ is Brownian representable.
\end{abstract}



\maketitle


\section{Introduction}

In the fundamental work \cite{Doob53}, Doob showed that a real-valued continuous 
local martingale can be represented as a stochastic integral with respect to a 
real-valued Brownian motion if and only if this local martingale has an 
absolutely continuous quadratic variation. Starting from this point, for a 
Banach space $X$ the following problem appears: find necessary and 
sufficient conditions for an $X$-valued local martingale $M$ in order that there exist a 
(cylindrical) Brownian motion $W$ and a stochastically integrable function $g$ such that $M = 
\int_0^{\cdot}g\ud W$ (we then call $M$ {\it Brownian representable}.) For some special
instances of Banach spaces $X$, Brownian representation results are well known. The finite-dimensional 
version was derived in \cite{KS,SV} and a generalization to the Hilbert space case 
was obtained using different techniques in \cite{DPZ, LO,Ouv1}. But for a general 
Banach space $X$ some problems arise. For instance, one cannot define a quadratic 
variation of a Banach space-valued continuous martingale in a proper way (see 
for example we refer the reader to \cite{DGFR},
where a notion of quadratic variation is defined which, however, is not
well-defined for some particular martingales) and therefore it seems difficult to find appropriate 
necessary conditions for an arbitrary martingale to be Brownian representable. 
Nevertheless, quite general sufficient conditions (which are also
necessary when the Banach space is a Hilbert space) were obtained by 
Dettweiler in \cite{Dett}.

In order to generalize the above-mentioned results one can work with so-called cylindrical 
continuous local martingales (see \cite{JKFR,Ond1,Ond2,MP,Sch96,VY}), defined 
for an arbitrary Banach space $X$ as continuous linear mappings from 
$X^*$ to a linear space of continuous local martingales $\mathcal M^{\rm loc}$ 
equipped with the ucp topology. Using this approach together with 
functional calculus arguments, Ondrej\'at \cite{Ond1, Ond2} has shown that
if $X$ is a reflexive Banach space, then a cylindrical continuous local 
martingale $M$ is Brownian representable under appropriate conditions. More precisely, he shows that 
in $M$ is Brownian representable if there exist a 
Hilbert space $H$ and a scalarly progressively measurable 
process $g$ with values in $\mathcal L(H,X)$ such that
\begin{equation}\label{eq:1.1}
\begin{split}
  \int_0^t\|g\|^2\ud s &<\infty\;\;\; a.s.\; \forall t>0,\\
 [Mx^*]_t &= \int_0^t\|g^*x^*\|^2\ud s\;\;\; a.s.\; \forall t>0.
\end{split}
\end{equation}
The definition of a quadratic variation of a cylindrical continuous local 
martingale given in \cite{VY} was inspired by this result.

In this work we show that if $[Mx^*]$ is absolutely 
continuous for each $x^* \in X^*$, then there exist a Hilbert space $H$, an 
$H$-cylindrical Brownian motion $W_H$ defined on an enlarged probability space 
with an enlarged filtration, and a progressively scalarly measurable process 
$G$ with values in space of (possibly unbounded) linear operators 
from $X^*$ to $H$ such that $Mx^*=\int_0^{\cdot}Gx^*\ud 
W_H$ for all $x^*\in X^*$. Moreover, necessary and sufficient conditions for a so-called weak 
Brownian representation of an arbitrary $X$-valued continuous local martingale 
are established.

\vspace{.2cm}

It is well known (see \cite{Kal}) that each $\mathbb R$-, $\mathbb R^d$-, or 
$H$-valued continuous local martingale $M$ admits a time change $\tau$ 
such that $M\circ\tau$ is Brownian representable. This assertion is tied up to the fact 
that $[M\circ \tau] = [M]\circ \tau$, which makes the choice of $\tau$ evident. 
The same can be easily shown for a cylindrical continuous local martingale with 
a quadratic variation (see \cite{VY}). Here we prove that such a time change exists 
for arbitrary cylindrical continuous local martingales.

\vspace{.2cm}

To conclude this introduction we would like to point out the techniques
developed by Ondrej\'at in \cite{Ond1,Ond2}, in particular results developing a bounded Borel 
calculus for bounded operator-valued functions.
Quite a reasonable part of the 
current article is dedicated to applying and extending these techniques to a 
closed operator-valued $g$. Of course statements as \eqref{eq:1.1} do not make 
sense then, but thanks to closability of $g$ it is still possible to prove
some results on stochastic integrability for such $g$.

\section{Preliminaries}

We denote $[0,\infty)$ by $\mathbb R_+$. 

The {\it Lebesgue-Stieltjes measure} of a function 
$F:\mathbb R_+\to \mathbb R$ of bounded variation is the finite Borel measure 
$\mu_F$ on $\mathbb R_+$ defined by $\mu_F([a,b)) = F(b)-F(a)$ for $0\leq 
a<b<\infty$.

Let $(S,\Sigma)$ be a measurable space and let $(\Omega, \mathcal F, \mathbb P)$ 
be a probability space.
A mapping $\nu:\Sigma\times \Omega\to [0,\infty]$ will be called a {\em random 
measure} if for all $A\in \Sigma$, $\omega\mapsto \nu(A,\omega)$ is measurable 
and for almost all $\omega\in \Omega$, $\nu(\cdot, \omega)$ is a measure on 
$(S,\Sigma)$ and $(S, \Sigma, \nu(\cdot, \omega))$ is separable (i.e. 
the corresponding $L^2$-space is separable).

Random measures arise naturally when working with continuous local martingales. 
Indeed, for almost all $\omega \in \Omega$, the quadratic variation process 
$[M](\cdot, \omega)$ of a continuous local martingale $M$ is continuous and 
increasing (see \cite{Kal,MP,Prot}), so we can associate $\mu_{[M]}(\cdot, 
\omega)$ with it.

Let $(S,\Sigma,\mu)$ be a measure space. Let $X$ and $Y$ be Banach spaces. An 
operator-valued function $f: S \to \mathcal L(X,Y)$ is called {\em scalarly 
measurable} if for all $x\in X$ and $y^*\in Y^*$ the function $s\mapsto \langle 
y^*, f(s)x\rangle$ is measurable. If $Y$ is separable, by the Pettis measurability
theorem this is equivalent to the strong measurability of  
$s\mapsto f(s)x$ for each $x\in X$ (see \cite{HNVW1}).

Often we will use the notation $A \lesssim_Q B$ to indicate that there exists a 
constant $C$ which only depends on the parameter(s) $Q$ such that $A\leq C B$.

\section{Results}

\subsection{Cylindrical martingales and stochastic integration}
In this section we assume that $X$ is a separable Banach space with a dual space 
$X^*$.
 Let $(\Omega, \mathcal F, \mathbb P)$ be a complete probability space with 
filtration
$\mathbb F := (\mathcal F_t)_{t \in \mathbb R_+}$ that satisfies the usual 
conditions,
and let $\mathcal F := \sigma(\bigcup_{t\geq 0} \mathcal F_t)$.
We denote the predictable $\sigma$-algebra by $\mathcal P$.

A scalar-valued process $M$ is called a {\it continuous local martingale} if there 
exists a sequence of stopping times $(\tau_n)_{n\geq 1}$ such that 
$\tau_n\uparrow \infty$ almost surely as $n\to \infty$ and $\one_{\tau_n>0} 
M^{\tau_n}$ is a continuous martingale.

Let $\mathcal M^{\rm loc}$ be the class of all continuous local martingales.
On $\mathcal M^{\rm loc}$ define the translation invariant metric given by
\begin{equation*}\label{eq:metric}
 \|M\|_{\mathcal M^{\rm loc}} = \sum_{n=1}^{\infty} 2^{-n}\mathbb E[1 \wedge 
\sup_{t\in [0,n]} |M|_t].
\end{equation*}
One can easily check that the topology generated by this metric coincides with 
the ucp topology (uniform convergence on compact sets in probability). Moreover, 
$M_n\to 0$ in $\mathcal M^{\rm loc}$ if and only if for every $T\geq 0$, 
$[M_n]_T\to 0$ in probability (see \cite[Proposition 17.6]{Kal}).

Let $X$ be a separable Banach space, $\mathcal M^{\rm loc}(X)$ be the space of 
$X$-valued continuous local martingales. $\mathcal M^{\rm loc}(X)$ is complete 
under the ucp topology generated by the following metric
\begin{equation*}
 \|M\|_{\mathcal M^{\rm loc}(X)} = \sum_{n=1}^{\infty} 2^{-n}\mathbb E[1 \wedge 
\sup_{t\in [0,n]} \|M\|_t].
\end{equation*}
The completeness can be proven in the same way as in \cite[Part 3.1]{VY}. 

If $H$ is a Hilbert space and $M \in \mathcal M^{\rm loc}(H)$, then we define 
the quadratic variation $[M]$ as a compensator of $\|M\|^2$, and one can show 
that a.s.\
$$
[M] = \sum_{n=1}^{\infty}[\langle M, h_n\rangle],
$$
where $(h_n)_{n=1}^{\infty}$ is any orthonormal basis of $H$. For more details 
we refer to \cite[Chapter 14.3]{MP}.

\begin{remark}
 One can show that convergence in the ucp topology on $\mathcal M^{\rm loc}(H)$ 
is equivalent to convergence of quadratic variation in the ucp topology. This 
fact can be shown analogously to the scalar case, see \cite[Proposition 
17.6]{Kal}.
\end{remark}

Let $X$ be a Banach space. A continuous linear mapping $M:X^* \to \mathcal 
M^{\rm loc}$ is called a {\it cylindrical continuous local martingale}.
We will write $M \in \mathcal M_{\rm cyl}^{\rm loc}(X)$. Details on cylindrical 
martingales can be found in \cite{JKFR,VY}.
For a cylindrical continuous local martingale $M$ and a stopping time $\tau$ we 
define $M^{\tau}:X^* \to \mathcal M^{\rm loc}$ by $M^{\tau}x^*(t) = 
Mx^*(t\wedge\tau)$. In this way $M^{\tau}\in \mathcal M_{\rm cyl}^{\rm loc}(X)$ 
again.

Let $Y$ be a Banach space such that there exists a continuous embedding 
$j:Y^*\hookrightarrow X^*$ (e.g. $X$ is densely embedded in $Y$). Then define 
$M|_{Y}:Y^*\to \mathcal M^{\rm loc}$ by $y^*\mapsto M(jy^*)$. Obviously 
$M|_{Y}\in \mathcal M_{\rm cyl}^{\rm loc}(Y)$.

\begin{example}[Cylindrical Brownian motion]\label{wiencyl}
Let $X$ be a Banach space and $Q\in \mathcal L(X^*,X)$ be a positive 
self-adjoint operator, i.e. $\forall x^*,y^*\in X^*$ it holds that $\langle x^*, 
Qy^*\rangle = \langle y^*, Qx^*\rangle$ and $\langle Qx^*,x^*\rangle\geq 0$.
Let $W^Q:\mathbb R_+ \times X^* \to L^0(\Omega)$ be a {\it cylindrical $Q$-Brownian 
motion} (see \cite[Chapter 4.1]{DPZ}), i.e.
\begin{itemize}
 \item $W^Q(\cdot)x^*$ is a Brownian motion for all $x^* \in X$,
 \item $\mathbb E W^Q(t)x^*\,W^Q(s)y^* = \langle Qx^*,y^*\rangle \min\{t,s\}$
 $\forall x^*,y^* \in X^*$, $t,s \geq 0$.
\end{itemize}
The operator $Q$ is called the {\it covariance operator} of $W^Q$.
(See more in \cite[Chapter~1]{Nua} or \cite[Chapter 4.1]{DPZ} for a Hilbert 
space valued case
and \cite[Chapter~5]{NW1} or \cite{Rie} for the general case).
Then $W^Q\in \mathcal M_{\rm cyl}^{\rm loc}(X)$.

If $X$ is a Hilbert space and $Q = I$ is the identity operator, we call $W_X := 
W^I$ an {\it $X$-cylindrical Brownian motion}.
\end{example}

Let $X,Y$ be Banach spaces, $x^*\in X^*$, $y\in Y$. We denote by $x^*\otimes 
y\in \mathcal L(X,Y)$ a rank-one operator that maps $x\in X$ to $\langle 
x,x^*\rangle y$.

\begin{remark}\label{rem:adjofelem}
Notice that the adjoint of a rank-one operator is again a rank-one operator and 
$(x^*\otimes y)^* = y\otimes x^*:Y^*\to X^*$. Also for any Banach space $Z$ and 
bounded operator $A:Y\to Z$ we have that $A(x^*\otimes y) = x^*\otimes (Ay)$.
\end{remark}

The process $f: \mathbb R_+ \times \Omega \to \mathcal L(X,Y)$ is called {\it 
elementary progressive}
with respect to the filtration $\mathbb F = (\mathcal F_t)_{t \in \mathbb R_+}$ 
if it is of the form
$$
f(t,\omega) = \sum_{n=1}^N\sum_{m=1}^M \mathbf 1_{(t_{n-1},t_n]\times 
B_{mn}}(t,\omega)
\sum_{k=1}^Kx^*_{k}\otimes y_{kmn},
$$
where $0 \leq t_0 < \ldots < t_n <\infty$, for each $n = 1,\ldots, N$,
$B_{1n},\ldots,B_{Mn}\in \mathcal F_{t_{n-1}}$, $(x^*_k)_{k=1}^K\subset X^*$ and 
$(y_{kmn})_{k,m,n=1}^{K,M,N}\subset Y$.
For each elementary progressive
$f$ we define the stochastic integral with respect to $M\in \mathcal M_{\rm 
cyl}^{\rm loc}(X)$
as an element of $L^0(\Omega; C_b(\mathbb R_+;Y))$:
\begin{equation}\label{intnorm}
 \int_0^t f(s) \ud M(s) = \sum_{n=1}^N\sum_{m=1}^M \mathbf 
1_{B_{mn}}\sum_{k=1}^K
(M(t_n\wedge t)x^*_{k} - M(t_{n-1}\wedge t)x^*_{k})y_{kmn}.
\end{equation}
Often we will write $f \cdot M$ for the process $\int_0^\cdot f(s) dM(s)$.

\begin{remark}
 Notice that the integral \eqref{intnorm} defines the same stochastic process 
for a different form of finite-rank operator $\sum_{k=1}^Kx^*_{k}\otimes 
y_{kmn}$. Indeed, let $(A_{mn})_{mn}$ be a set of operators from $X$ to $Y$ such 
that $(A_{mn})_{mn} = (\sum_{k=1}^Kx^*_{k}\otimes y_{kmn})_{mn}$. Then $f\cdot 
M$ takes its values in a finite dimensional subspace of $Y$ depending only on 
$(A_{mn})_{mn}$. Let $Y_0 = \text{span}(\text{ran }(A_{mn}))_{mn}$.
For each fixed 
$y^*_0\in Y_0^*$, $m$ and $n$ one can define $\ell_{mn}\in X^*$, $x\mapsto 
\langle y_0^*,A_{mn}x \rangle$. In particular $\ell_{mn}(x) = 
\sum_{k=1}^K\langle x^*_{k}\langle y^*_0,y_{kmn}\rangle,x\rangle$. Then because 
of linearity of $M$
 \begin{align*}
 \langle y_0^*,\int_0^t f(s) \ud M(s)\rangle &= \sum_{n=1}^N\sum_{m=1}^M \mathbf 
1_{B_{mn}}\sum_{k=1}^K
(M(t_n\wedge t)x^*_{k} - M(t_{n-1}\wedge t)x^*_{k})\langle 
y^*_0,y_{kmn}\rangle\\
&= \sum_{n=1}^N\sum_{m=1}^M \mathbf 1_{B_{mn}}
(M(t_n\wedge t) - M(t_{n-1}\wedge t))\Bigl(\sum_{k=1}^K x^*_{k}\langle 
y^*_0,y_{kmn}\rangle\Bigr)\\
&= \sum_{n=1}^N\sum_{m=1}^M \mathbf 1_{B_{mn}}
(M(t_n\wedge t) - M(t_{n-1}\wedge t))\ell_{mn},
\end{align*}
where the last expression does not depend on the form of $(A_{mn})_{mn}$. Then 
since $Y_0$ is finite dimensional, the entire integral does not depend on the 
form of $(A_{mn})_{mn}$.
\end{remark}

\vspace{.2cm}

We say that
$M \in \mathcal M_{\rm cyl}^{\rm loc}(X)$ is {\it Brownian representable} if 
there exist a Hilbert space $H$, an $H$-cylindrical Wiener process $W_H$ on an 
enlarged probability space
$(\overline{\Omega},\overline{ \mathbb F}, \overline{\mathbb P})$ and
$g:\mathbb R_+\times\overline{\Omega}\times X^*\to H$ such that $g(x^*)$ is 
$\overline{\mathbb F}$-scalarly
progressively measurable and a.s.\
\begin{equation}\label{eq:brrepr}
 Mx^* = \int_0^{\cdot} g^*(x^*) dW_H, \;\;\; x^* \in X^*.
\end{equation}

We call that $M \in \mathcal M_{\rm cyl}^{\rm loc}(X)$
is {\it with an absolutely continuous covariation} (or $M \in \mathcal M_{\rm 
a.c.c.}^{\rm cyl}(X)$) if for each $x^*,y^* \in X^*$ the covariation 
$[Mx^*,My^*]$ is absolutely continuous a.s.\ Note that by 
\cite[Proposition~17.9]{Kal} this is equivalent to $[Mx^*]$ having an 
absolutely continuous version for all $x^*\in X^*$.

\subsection{Closed operator representation}
In view of the results of \cite[Theorem~2]{Ond1}, the natural problem presents itself to 
extend this theorem to the case of an unbounded operator-valued function $g$. 
Such a generalization will be proven in the present subsection.

\vspace{.05cm}

 For Banach spaces $X,Y$ define $\mathcal L_{cl}(X,Y)$ as a set of all closed
densely defined operators from $X$ to $Y$, $\mathcal L_{cl}(X):=\mathcal 
L_{cl}(X,X)$.

\vspace{.05cm}

Let $M\in \mathcal M_{\rm cyl}^{\rm loc}(X)$ be Brownian representable. We say 
that $M$ is {\it closed operator-Brownian representable} (or simply $M\in 
\mathcal M_{\rm cl}^{\rm cyl}(X)$) if there exists $G:\mathbb R_+ \times \Omega 
\to \mathcal L_{cl}(H,X)$ such that for each fixed $x^* \in X^*$ $G^*x^*$ is 
defined $\mathbb P\times \ud s$-a.s.\ and it is a version of corresponding 
$g^*(x^*)$ from \eqref{eq:brrepr}.

\begin{definition}
  Let $X$ be a Banach space with separable dual, $M\in \mathcal M_{\rm cyl}^{\rm 
loc}(X)$. Then $M$ has {\it a closed operator-generated covariation} if there 
exist a separable Hilbert space $H$ and a closed operator-valued function 
$G:\mathbb R_+\times\Omega \to \mathcal L_{cl}(H,X)$ such that for all $x^*\in 
X^*$, $G^*x^*$ is defined $\mathbb P\times\ud s$-a.s.\ and progressively 
measurable, and for all $x^*,y^*\in X^*$ a.s.\
 \begin{equation}\label{theorembrownianrepresenrationeq1}
  [ Mx^*,My^*]_t = \int_0^t \langle G^*x^*, G^*y^*\rangle \ud s,\;\;\; t\geq 0.
 \end{equation}
\end{definition}

Notice that the last assumption is equivalent to the fact that 
$[Mx^*]=\int_0^{\cdot}\|G^*x^*\|^2\ud s$ a.s.\ for all $x^* \in X^*$.

\begin{proposition}\label{stochintabscont}
 Let $H$ be a separable Hilbert space, $M\in \mathcal M_{\rm cyl}^{\rm loc}(H)$ 
has a closed operator-generated covariation.
 Let $G:\mathbb R_+\times\Omega \to \mathcal L_{cl}(H)$ be the corresponding 
covariation family.
 Then for each scalarly progressively measurable $f :\mathbb R_+\times\Omega \to 
H$ such that
 $G^*f$ is defined $\mathbb P \times \ud s$-a.e.\ and
 $\|G^*f\|\in L^2_{\rm loc}(\mathbb R_+)$ a.s.\ one can define the stochastic 
integral $\int_0^{\cdot}fdM$,
 and a.s.
 \begin{equation}\label{quadvarint}
 [\int_0^{\cdot}fdM]_{T} = \int_0^T\langle G^*(s)f(s),G^*(s)f(s)\rangle 
ds,\;\;\;T>0.
\end{equation}
\end{proposition}

\begin{proof}
 Applying the stopping time argument one can restrict the proof to the case
 $\mathbb E[\int_0^{\infty}\langle G^*f,G^*f\rangle ds] < \infty$. First of all 
it is easy to construct
 the stochastic integral if $\text{ran }(f)\subset H_0$, where $H_0$ is a fixed 
finite dimensional subspace of $H$,
 since in this case by redefining $G^*$ one can assume that $G^*h$ is 
well-defined for all $h\in H_0$,
so we just work with a finite dimensional martingale for which 
\eqref{quadvarint} obviously holds according to the isometry 
\cite[Chapter~14.6]{MP}.

 The general case can be constructed in the following way. Let $(h_i)_{i\geq 1}$ 
be an orthonormal basis
 of $H$. For each $k\geq 1$ set $H_k = \text{span}(h_1,\ldots,h_k)$. Then by 
Lemma \ref{technical} one can construct scalarly progressively measurable 
$\tilde P_k:\mathbb R_+\times\Omega
 \to \mathcal L (H)$, which is an orthogonal projection onto $G^* (H_k)$, and a 
scalarly progressively measurable $L_k:\mathbb R_+\times\Omega
 \to \mathcal L (H,H_k)$ such that $G^*L_k = \tilde P_k$.
 Let $f_k = L_k G^* f$. Then from \eqref{quadvarint} for $f_k$, the fact that 
$\|G^*f_k\| = \|\tilde P_kG^*f\|\nearrow \|G^*f\|$ and $\|G^*f_k - G^*f\|\to 0$ 
$\mathbb P\times \ud s$ a.s.,
 dominated convergence theorem, \cite[Proposition~17.6]{Kal} and the fact that 
$\text{ran }f_k \subset H_k$ one can construct stochastic integral
 $\int_0^{\cdot}fdM$ and \eqref{quadvarint} holds true.
\end{proof}

\begin{remark}\label{rem:extrem31}
 Using the previous theorem one can slightly extend \cite[Remark 31]{Ond1} in 
the following way: let $\Psi:\mathbb R_+\times\Omega \to \mathcal L(H)$ be 
progressively scalarly measurable such that $G^*\Psi^*\in \mathcal L(H)$ a.s.\ 
for all $t\geq 0$ and $\|G^*\Psi^*\|\in L^2_{\rm loc}(\mathbb R_+)$ a.s. Then 
one can define $N\in \mathcal M_{\rm cyl}^{\rm loc}(H)$ as follows:
 $$
 Nh = \int_0^{\cdot} \Psi^* h \ud M.
 $$
 Moreover, then for each progressively measurable $\phi:\mathbb 
R_+\times\Omega\to H$ such that $\|G^*\Psi^*\phi\|\in L^2_{\rm loc}(\mathbb 
R_+)$ a.s. one has that
 \begin{equation}\label{eq:rem31}
   \int_0^{\cdot}\phi \ud N = \int_0^{\cdot}\Psi^* \phi \ud M.
 \end{equation}
\end{remark}

\begin{theorem}\label{theorembrownianrepresenration}
 Let $X$ be a separable reflexive Banach space, $M \in \mathcal M_{\rm cyl}^{\rm 
loc}(X)$. Then $M$ is closed operator-Brownian representable if and only if it 
has a closed operator-generated covariation.
\end{theorem}

Let $n\geq 1$, $X_1,\ldots,X_n$ be Banach spaces, 
$A_{k}\in\mathcal L_{cl}(X_k,X_{k+1})$ for $1\leq k\leq n-1$. Then we say that 
$A_{n-1}\ldots A_1$ is {\it well-defined} if $\text{ran }(A_{k-1}\ldots A_1) \in 
\text{dom } (A_k)$ for each $2\leq k\leq n-1$.

\vspace{.1cm}

\begin{proof}
Suppose that $M$ is closed operator-Brownian representable. Let a separable 
Hilbert space $H$, an $H$-cylindrical Brownian motion $W_H$ and $G: \mathbb R_+ 
\times \Omega \to \mathcal L_{cl}(H,X)$ be such that a.s.\ $Mx^*= \int_0^{\cdot} 
G^*x^*dW_H$ for each $x^*\in X^*$. Then according to \cite[Theorem 4.27]{DPZ} 
a.s.\
$$
[Mx^*,My^*] = \int_0^{\cdot}\langle G^*x^*, G^*y^*\rangle\ud s,\;\;\; x^*,y^*\in 
X^*.
$$

 To prove the other direction assume that there exist such a separable Hilbert 
space $H$ and $G: \mathbb R_+ \times \Omega \to \mathcal L_{cl}(H,X)$ that 
\eqref{theorembrownianrepresenrationeq1} holds. The proof that
 $M$ is Brownian representable will be almost the same as the proof of 
\cite[Theorem~2]{Ond1}, but one
 has to use Lemma \ref{Borelcalculus1} instead of \cite[Proposition~32]{Ond1} 
and apply
 \cite[Proposition~3.6]{Ond2} for general Banach spaces.

 Suppose first that $X$ is a Hilbert space (one then can identify $H$, $X$ and 
$X^*$). Let $\overline {W}_{X}$ be an independent of $M$ $X$-cylindrical Wiener 
process on an enlarged
 probability space $(\overline {\Omega}, \overline{\mathbb F}, \overline{\mathbb 
P})$ with the enlarged
 filtration $\overline{\mathbb F} = (\overline {\mathcal F}_t)_{t\geq 0}$. Let 
$(0,\infty) = \cup_{n=1}^{\infty} B_n$ be
a decomposition into disjoint Borel sets such that $\text{dist}(B_n,\{0\})>0$ 
for each $n\geq 1$. Define
functions $\psi_n(t)=t^{-1}\mathbf 1_{B_n}$, $t\in\mathbb R$, $n\geq 1$, and 
$\psi_0 = \mathbf 1_{\{0\}}$.
Let us also denote $C_n = \mathbf 1_{B_n}$, $n\geq 1$, and $C_0 = \psi_0$. By 
Lemma \ref{Borelcalculus1} for each $n\geq 1$
$\psi_n(G^*G)$ and $C_n(G^*G)$ are $\mathcal L(X)$-valued strongly progressively 
measurable,
and since
$$
\|G^*G\psi_n(G^*G)\|_{\mathcal L(X)} = \|C_n(G^*G)\|_{\mathcal L(X)} \leq 
\|\mathbf 1_{B_n}\|_{L^{\infty}(\mathbb R)} = 1,\;\;\; n\geq 1,
$$
$$
\|\psi_0(G^*G)\|_{\mathcal L(X)}\leq \|\mathbf 1_{\{0\}}\|_{L^{\infty}(\mathbb 
R)} = 1,
$$
then $\text{ran }(\psi_n(G^*G))\subset \text{dom } (G^*G)\subset \text{dom } (G)$ (see 
\cite[p.347]{RSN}),
so $G\psi_n(G^*G)$ is well-defined, and according to \cite[Problem 
III.5.22]{Kato1} $G\psi_n(G^*G) \in \mathcal L(X)$ $\mathbb P\times \ud s$-a.s., 
therefore for each $x^*\in X^*$ by Proposition \ref{stochintabscont} 
$G\psi_n(G^*G)x^*$ is
stochastically integrable with respect to $M$ and one can define
$$
W_n(x^*) = \int_0^{\cdot}G\psi_n(G^*G)x^*\ud M,\;\;\;n\geq 1,\;\;\; W_0 = 
\int_0^{\cdot}\psi_0(G^*G)x^*\ud\overline W_X.
$$

Then by \eqref{quadvarint} $[W_n(x^*),W_m(y^*)]=0$ for each $x^*,y^* \in X^*$ 
and $0\leq n<m$ and
$[W_n (x^*)] = \int_0^{\cdot} \|C_n(G^*(s)G(s))x^*\|^2\ud s$, so for each $t\geq 
0$
\begin{align*}
\sum_{n=0}^{\infty}[W_n (x^*)]_t =\sum_{n=0}^{\infty} 
\int_0^t\|C_n(G^*G)x^*\|^2\ud s &= \sum_{n=0}^{\infty} \int_0^t\langle 
C_n(G^*G)x^*,x^*\rangle \ud s\\
&= \int_0^t\langle x^*,x^*\rangle ds = t\|x^*\|^2.
\end{align*}

Let us define $W(x^*) = \sum_{n=0}^{\infty}W_n(x^*)$. Thanks to 
\cite[Proposition 28]{Ond1} this sum converges in $C([0,\infty))$ in 
probability. It is obvious that $W:X^*\to \mathcal M^{\rm loc}$, $x^*\mapsto 
W(x^*)$ is an $H$-cylindrical Brownian motion. Moreover, for each $h\in X^*$ one 
has that by the definition of $G^*$ the $H$-valued function $G^*h$ is 
stochastically integrable with respect to $W_H$, and $[G^*h\cdot W] = [Mh] = 
\int_0^{\cdot} G^*(s)h\ud s$ a.s. So, to prove that $Mh$ and $G^*h\cdot W$ are 
indistinguishable it is enough to show that a.s.\ $[G^*h\cdot 
W,Mh]=\int_0^{\cdot} G^*(s)h\ud s$. By Remark \ref{rem:extrem31} and the fact 
that a.s.\ $\|G^*G\psi_n(G^*G)\|\leq \mathbf 1_{\mathbb R_+}\in L^2_{\rm 
loc}(\mathbb R_+)$ one can apply \eqref{eq:rem31} so for each $n\geq 1$
\begin{align*}
  [\int_0^{\cdot} G^*h\ud W_n,Mh]&=[\int_0^{\cdot} G\psi_n(G^*G)G^*h \ud 
M,\int_0^{\cdot} h\ud M]\\
  &=\int_0^{\cdot}\langle G^*G\psi_n(G^*G)G^*h,G^*h\rangle\ud s.
\end{align*}
On the other hand $\widetilde{W}$ and $M$ are independent so $[G^*h\cdot 
W_0,Mh]=0$. To sum up one has that a.s.\
\begin{align*}
 [G^*h\cdot W,Mh] = \sum_{n=1}^{\infty}\int_0^{\cdot}\langle 
C_n(G^*G)G^*h,G^*h\rangle
 =\int_0^{\cdot}\|G^*h\|^2\ud s,
\end{align*}
which finishes the proof for a Hilbert space case.

 Now consider a general reflexive Banach space $X$. Let a separable Hilbert 
space $H$, $j:X \to H$ be defined as in \cite[p.154]{Kuo}. Let $W_H$ be 
constructed as above but for $M|_{H}$. Fix $x^*\in X^* $ and find
 $(h_n)_{n\geq 1},\subset H$ such that $\lim_{n\to \infty}j^*h_n=x^*$ and 
$[M(j^*h_n)-M(x^*)]_T$ vanishes almost everywhere
 for each $T>0$. By the definition of $G$ one has $\int_0^T 
\|G^*(j^*h_n-x^*)\|^2\ud s\to 0$ in probability for all $T>0$,
 so since $M(j^*h_n) \to M(x^*)$ uniformly on all compacts in probability and by 
\cite[Theorem~4.27]{DPZ}, \cite[Proposition~17.6]{Kal} and Proposition 
\ref{stochintabscont} $M(x^*)$ and $G^*x^*\cdot W_H$ are indistinguishable.
\end{proof}

Thanks to this representation theorem we obtain the following  
stochastic integrability result.

\begin{theorem}\label{theorem:inttheory}
Let $X$ be a reflexive separable Banach space, $M\in \mathcal M_{\rm cyl}^{\rm 
loc}(X)$ be closed operator-Brownian representable, $G:\mathbb R_+\times\Omega 
\to \mathcal L_{cl}(H,X)$ be the corresponding operator family. Let $f:\mathbb 
R_+\times \Omega \to X^*$. Suppose there exist elementary progressive 
$f_n:\mathbb R_+\times \Omega \to X^*$, $n\geq 1$, such that $f_n \to f$ 
$\mathbb P\times \ud s$-a.s.\ Assume also that there exists a limit $N := 
\lim_{n\to \infty}f_n \cdot M$ in the ucp topology. Then $f\in \text{ran } (G^*)$ 
$\mathbb P\times \ud s$-a.s., $G^*f$ is progressively measurable and
\begin{equation}\label{eq:quadvarforint}
[N] = \int_0^{\cdot} \|G^*f\|^2\ud s.
\end{equation}
\end{theorem}
We then call $f$ stochastically integrable and define
$$
f\cdot M = \int_0^{\cdot}f\ud M:=\lim_{n\to \infty}f_n\cdot M,
$$
where the limit is taken in the ucp topology.

\begin{proof}
Formula \eqref{eq:quadvarforint} is obvious for elementary progressive $f$ by 
\eqref{quadvarint}. Since $[f_n\cdot M]$ is absolutely continuous and $f_n\cdot 
M$ tends to $N$ in ucp, by Lemma \ref{lemma:appB1} and the fact that $[N]$ is 
a.s.\ continuous one can prove that $[N]$ is absolutely continuous, hence by 
\cite[Lemma~3.10]{VY} its derivative in time $v:\mathbb R_+\times \Omega \to 
\mathbb R$ is progressively measurable. Let a Hilbert space $H$ and an 
$H$-cylindrical Brownian motion $W_H$ be constructed for $M$ by Theorem 
\ref{theorembrownianrepresenration}. Assume that $(h_n)_{n\geq 1}\subset H$ is a 
dense subset of a unite ball in $H$. Then $[N, W_H h_n]$ is a.s.\ absolutely 
continuous, and has a progressively measurable derivative $v_n$ for each $n\geq 
1$. Moreover, since by \cite[Proposition 17.9]{Kal} $|\mu_{[N, W_H 
h_n]}|(I)\leq\mu_{[N]}^{1/2}(I)\mu_{[W_H h_n]}^{1/2}(I)$ for each interval $I 
\subset \mathbb R_+$ and thanks to \cite[Theorem~5.8.8]{Bog} one has that $v_n 
\leq v^{1/2}\|h_n\|$ $\mathbb P\times\ud s$-a.s.\ Then thanks to linearity, boundedness and denseness 
of $\text{span}(h_n)_{n\geq 1}$ in $H$, we obtain that there exists a 
progressively measurable process $V:\mathbb R_+\times\Omega \to H$ such that 
$v_n = \langle V,h_n\rangle$ $\mathbb P\times\ud s$-a.s.\ Let $\tilde N = V 
\cdot W_H$. Then $[ N, \Phi\cdot W_H] = [\tilde N, \Phi\cdot W_H]$ for each 
stochastically integrable $\Phi:\mathbb R_+ \times \Omega \to H$, hence 
$[N,\tilde N] = [\tilde N]$ and $[ N, f_n\cdot M] = [\tilde N, f_n\cdot M]$ for 
each natural $n$. Without loss of generality one can suppose that $\lim_{n\to 
\infty}[N-f_n\cdot M]_T = 0$ a.s.\ for each $T>0$, therefore a.s.\
\begin{align*}
\lim_{n\to \infty}\int_0^T \|G^*f_n - V\|^2\ud s &= \lim_{n\to \infty}[\tilde N 
- f_n\cdot M]_T \\
&=\lim_{n\to \infty} ([\tilde N]_T - [\tilde N,f_n\cdot M]_T)+([f_n\cdot M]_T - 
[ N, f_n\cdot M]_T)\\
&=([\tilde N]_T - [\tilde N,N]_T)+([N]_T - [N]_T) = 0,
\end{align*}
so, $\tilde N$ is a version of $N$. Also by choosing a subsequence one has that 
$G^*f_{n_k}\to V$ as $k\to \infty$, which means that $f \in \text{dom } (G^*)$ 
$\mathbb P\times\ud s$-a.s.\ and \eqref{eq:quadvarforint} holds.
\end{proof}

\begin{remark}\label{rem:afterumdthm}
 It follows from Theorem \ref{theorem:inttheory} and \cite[Proposition 
17.6]{Kal} that for any finite dimensional subspace $Y_0\subset Y$ the 
definition of the stochastic integral can be extended to all strongly 
progressively measurable processes $\Phi:\mathbb R_+\times \Omega \to \mathcal 
L(X,Y_0)$ that satisfy $(G^*\Phi^*)^*\in L^2(\mathbb R_+;\mathcal L(H,Y_0))$ 
a.s.\ (or equivalently $(G^*\Phi^*)^*$ is scalarly in $L^2(\mathbb R_+;H)$ 
a.s.). Moreover, then $\Phi\cdot M = (G^*\Phi^*)^* \cdot W_H$.
\end{remark}

We proceed with a result which is closely related to \cite[Theorem 3.6]{NVW}. 
In order to state it we need the following terminology. 

A Banach space $X$ is called a {\it UMD Banach space} if for some (or 
equivalently, for all)
$p \in (1,\infty)$ there exists a constant $\beta>0$ such that
for every $n \geq 1$, every martingale
difference sequence $(d_j)^n_{j=1}$ in $L^p(\Omega; X)$, and every $\{-1, 
1\}$-valued sequence
$(\varepsilon_j)^n_{j=1}$
we have
$$
\Bigl(\mathbb E \Bigl\| \sum^n_{j=1} \varepsilon_j d_j\Bigr\|^p\Bigr )^{\frac 
1p}
\leq \beta \Bigl(\mathbb E \Bigl \| \sum^n_{j=1}d_j\Bigr\|^p\Bigr )^{\frac 1p}.
$$
The infimum over all admissible constants $\beta$ is denoted by $\beta_{p,X}$. 

There is a large body of results asserting that the class of UMD Banach spaces
is a natural one when pursuing vector-valued generalizations of scalar-valued 
results in harmonic and stochastic analysis. UMD spaces enjoy  
many pleasant properties, among them being reflexive. We 
refer the reader to \cite{Burk01,HNVW1,Rubio86} for details.

\vspace{.1cm}

Let $(\gamma_n')_{n \geq 1}$ be a sequence of independent standard Gaussian 
random variables
on a probability space $(\Omega', \mathcal F', \mathbb P')$ and let $H$ be a 
Hilbert space. A bounded operator $R \in\mathcal L (H, X)$ is said to be 
{\it $\gamma$-radonifying}
if for some (or equivalently for each) orthonormal basis $(h_n)_{n\geq 1}$ of 
$H$ the Gaussian series
$\sum_{n\geq 1} \gamma_n' Rh_n$
converges in $L^2 (\Omega'; X)$. We then define
$$
\|R\|_{\gamma(H,X)} :=\Bigl(\mathbb E'\Bigl\|\sum_{n\geq 1} \gamma_n' Rh_n 
\Bigr\|^2\Bigr)^{\frac 12}.
$$
This number does not depend on the sequence $(\gamma_n')_{n\geq 1}$ and the 
basis $(h_n)_{n\geq 1}$, and
defines a norm on the space $\gamma(H, X)$ of all $\gamma$-radonifying operators 
from $H$ into
$X$. Endowed with this norm, $\gamma(H, X)$ is a Banach space, which is 
separable if $X$
is separable. For a Hilbert space $X$, the space  $\gamma(H, X)$
is isometrically isomorphic to the space of all Hilbert-Schmidt 
operators from $H$ to $X$. If $(S,\mathcal A, \mu)$ is a measure space and $X = L^p(S)$, then $\gamma(H,L^p(S)) = L^p(S;H)$ up to equivalence of norms.

For all $R \in\gamma(H, X)$ it holds that $\|R\| \leq \|R\|_{\gamma(H, X)}$.
Let $\overline{H}$ be another Hilbert space and let $Y$ be another Banach space. Then 
the so-called {\it ideal property}
(see \cite{HNVW2}) holds true: for all $S \in\mathcal 
L(\overline{H},H)$ and all $T \in \mathcal L(X,Y)$
we have $TRS \in \gamma(\overline{H},Y)$ and
\begin{equation}\label{ideal}
\|TRS\|_{\gamma(\overline{H},Y)}\leq \|T\| \|R\|_{\gamma(H,X)}\|S\|.
\end{equation}

Let $X,Y$ be Banach spaces, and let $A\in \mathcal L_{cl}(X,Y)$. A linear 
subspace $X_0\subset \text{dom }(A)$ is a {\it core} 
of $A$ if the closure of an 
operator $A|_{X_0}:X_0 \to Y$ is $A$ (see details in \cite{Kato1}). Let 
$(S,\Sigma,\mu)$ be a measure space, and let $F:S\to\mathcal L_{cl}(X,Y)$ be such that 
$Fx$ is a.s.\ defined and measurable for each $x\in X$. Then $F$ has a {\it 
fixed core} if there exists a sequence $(x_n)_{n\geq 1}\subset X$ such that 
$\text{span}(x_n)_{n\geq 1}$ is a core of $F$ a.s.\ (Notice, that in this 
particular case the core has a countable algebraic dimension).

\begin{theorem}\label{thm:closedcovmain}
Let $X$ be a reflexive Banach space, $Y$ be a UMD Banach space, $M \in \mathcal 
M_{\rm cl}^{\rm cyl}(X)$, $G:\mathbb R_+\times \Omega \to \mathcal L_{cl}(H,X)$ 
be the corresponding operator family. Let $G^*$ have a fixed core. Then for a 
strongly progressively measurable process
$\Phi \colon \mathbb R_+ \times \Omega\rightarrow \mathcal L(X,Y)$
such that $G^*\Phi^*\in \mathcal L(Y^*,H)$ a.s.\ and which is scalarly in $ 
L^2(\mathbb R_+;H)$~a.s.\
the following assertions are equivalent:
\begin{enumerate}

 \item[(1)] There exist  elementary progressive processes $(\Phi_n)_{n \geq 1}$ 
such that:
 \begin{enumerate}
  \item[(i)] for all $y^* \in Y^*$, $\displaystyle \limn G^*\Phi_n^*y^* =
  G^*\Phi^*y^*$ in $L^0(\Omega; L^2(\mathbb R_+;H))$;
  \item[(ii)] there exists a process $\zeta \in L^0(\Omega; C_b(\mathbb R_+; 
Y))$ such that
  $$
  \zeta = \lim_{n \to \infty} \int_0^{\cdot} \Phi_n(t)\ud M(t) \;\;\; 
\text{in}\; L^0(\Omega; C_b(\mathbb R_+;Y)).
  $$
 \end{enumerate}
 \item[(2)] There exists an a.s.\ bounded process $\zeta\colon \mathbb R_+\times 
\Omega\to Y$ such that for all $y^* \in Y^*$
 we have
  $$
  \langle\zeta,y^*\rangle = \int_0^{\cdot}\Phi^*(t)y^* \ud M(t) \;\;\;
  \text{in}\: L^0(\Omega; C_b(\mathbb R_+)).
  $$

 \item[(3)] $(G^*\Phi^*)^* \in \gamma(L^2(\mathbb R_+;H), Y)$ almost surely;
\end{enumerate}
In this case $\zeta$ in (1) and (2) coincide and for all $p \in (0, \infty)$ we 
have
\begin{equation}\label{eq:equiv}
\mathbb E \sup_{t \in \mathbb R_+}\|\zeta (t)\|^p \eqsim_{p,Y} \mathbb E 
\|(G^*\Phi^*)^*\|^p_{\gamma (L^2(\mathbb R_+;H), Y)}.
\end{equation}
\end{theorem}

\begin{proof}
Let $\Psi:\mathbb R_+\times\Omega\to\mathcal L(H,Y)$ be such that $\Psi^* = 
G^*\Phi^*$ (recall that UMD spaces are reflexive). Then equivalence of $(2)$ and 
$(3)$ and the formula \eqref{eq:equiv} are just particular cases of 
corresponding parts of \cite[Theorem 5.9]{NVW} and \cite[Theorem 5.12]{NVW} for 
$\Psi$ thanks to Remark \ref{rem:extrem31} (for $(2)$ one also has to apply 
Theorem \ref{theorem:inttheory}).

It remains to prove that $(1)$ for $\Phi$ and $M$ and $(1)$ for $\Psi$ and $W_H$ 
are equivalent. (Notation: $(1,\Phi)\Leftrightarrow (1,\Psi)$).

$(1,\Phi)\Rightarrow(1,\Psi)$: Since $\Phi_n\cdot M$ exists and the range of 
$\Phi_n$ is in a certain fixed finite dimensional space of $Y$, $n\geq 1$, then 
by Remark \ref{rem:afterumdthm} $(G^*\Phi_n^*)^*$ is stochastically integrable 
with respect to $W_H$, so by \cite[Theorem 5.9]{NVW} there exists a sequence 
$(\Psi_{nk})_{k\geq 1}$ of elementary progressive $\mathcal L(H,Y)$-valued 
functions such that $\Psi_{nk}^*y^* \to G^*\Phi_n^*y^*$ in $L^0(\Omega; 
L^2(\mathbb R_+;H))$ for each $y^*\in Y^*$, and $\Psi_{nk}\cdot W_H \to 
\Phi_n\cdot M$ in $L^0(\Omega;C_b(\mathbb R_+;Y))$ as $k\to \infty$. Knowing 
$(1,\Phi)$ one can then find a subsequence $\{\Psi_k\}_{k\geq1}:= 
\{\Psi_{n_kk}\}_{k\geq1}$ such that $\lim_{k\to\infty}\Psi_k^*y^*=G^*\Phi^*y^* = 
\Psi^*y^*$ in $L^0(\Omega; L^2(\mathbb R_+;H))$ for each $y^*\in Y^*$ and 
$\Psi_{k}\cdot W_H$ converges in $L^0(\Omega;C_b(\mathbb R_+;Y))$, which is 
$(1)$ for $\Psi$ and $W_H$.

$(1,\Psi)\Rightarrow(1,\Phi)$: Let $(x_k^*)_{k\geq 1}\subset X^*$ be such that 
$U := \text{span }(x_k^*)_{k\geq 1}$ is a fixed core of $G^*$. For each $k\geq 
1$ define $U_k = \text{span}(x_1^*,\ldots,x_k^*)$. Then due to Lemma 
\ref{technical} consider $\tilde P_k:\mathbb R_+\times \Omega \to \mathcal L(H)$ 
and $L_k:\mathbb R_+\times \Omega \to \mathcal L(H,X^*)$ such that $\tilde P_k$ 
is an orthogonal projection onto $G^*(U_k)$ and $G^*L_k = \tilde P_k$ $\mathbb 
P\times \ud s$-a.s.\ Consider a sequence $(\Psi_n)_{n\geq 1}$ of elementary 
progressive functions in $L^0(\Omega, \gamma(L^2(\mathbb R_+;H),Y))$ constructed 
thanks to \cite[Proposition 2.12]{NVW} such that $\Psi_n \to \Psi$ in  
$L^0(\Omega, \gamma(L^2(\mathbb R_+;H),Y))$.
Let $\tilde P_0:\mathbb R_+\times\Omega\to \mathcal L(H)$ be an orthogonal 
projection onto $\overline{\text{ran } (G^*)}$. Then by the ideal property 
\eqref{ideal}, a.s. we have\
\begin{align*}
 \|(G^*\Phi^*)^* -\Psi_n\tilde P_0 \|_{\gamma(L^2(\mathbb 
R_+;H),Y)}&=\|((G^*\Phi^*)^* - \Psi_n)\tilde P_0 \|_{\gamma(L^2(\mathbb 
R_+;H),Y)}\\
 &\leq \|(G^*\Phi^*)^* -\Psi_n\|_{\gamma(L^2(\mathbb R_+;H),Y)},
\end{align*}
which means that by \cite[Theorem 3.6]{NVW} $\Psi_n\tilde P_0$ is stochastically 
integrable with respect to $W_H$ and $\Psi_n\tilde P_0\cdot W_H \to 
(G^*\Phi^*)^*\cdot W_H$ in $L^0(\Omega, \gamma(L^2(\mathbb R_+;H),Y))$.
Therefore one 
can define $\widetilde{\Psi}_n:=(\tilde P_0 \Psi_n^*)^*$.

Set $\Phi_{nk}:= (L_k\tilde P_k \widetilde{\Psi}_n^*)^*$. Notice that $U =\cup_k 
U_k$ is a core of $G^*$, so $\overline {\text{ran } (G^*)} = \overline {G^*(U)} $, 
therefore $\tilde P_k \to \tilde P_0$ weakly and by \cite[Proposition~2.4]{NVW} 
$(G^*\Phi_{nk}^*)^* = (\tilde P_k \widetilde{\Psi}_n^*)^* \to 
\widetilde{\Psi}_{n}$ in $L^0(\Omega, \gamma(L^2(\mathbb R_+;H),Y))$ as $k\to 
\infty$, so one can find a subsequence $\Phi_{n} := \Phi_{nk_n}$ such that 
$(G^*\Phi_{n}^*)^*  \to \Psi$ in $L^0(\Omega, \gamma(L^2(\mathbb R_+;H),Y))$ as 
$n\to \infty$. Now fix $n\geq 1$. Since $\Psi_n$ is elementary progressive, then 
it has the following form: for each $t\geq 0$ and $\omega \in \Omega$
$$
\Psi_n (t,\omega) = \sum_{m=1}^M\sum_{l=1}^L \mathbf 1_{(t_{m-1},t_m]\times 
B_{lm}}(t,\omega)
\sum_{j=1}^Jh_{j}\otimes y_{jlm}.
$$
Hence by Remark \ref{rem:adjofelem}
$$
\Phi_n (t,\omega) = \sum_{m=1}^M\sum_{l=1}^L \mathbf 1_{(t_{m-1},t_m]\times 
B_{lm}}(t,\omega)
\sum_{j=1}^J(L_{k_n}\tilde P_{k_n}h_{j})\otimes y_{jlm}.
$$
Therefore $\Phi_n$ takes it values in a fixed finite dimensional subspace $Y_n$ 
of $Y$, and so by Remark \ref{rem:afterumdthm} one can construct simple 
approximations of $\Phi_n$ then $(1,\Phi_n)$ holds. This completes the proof.

\end{proof}

\begin{remark}
 As the reader can see, the existence of a fixed core of $G^*$ is needed only 
for (1) in Theorem \ref{thm:closedcovmain}. Without this condition one can still 
show that parts (2) and (3) are equivalent and that estimate \eqref{eq:equiv} 
holds.
\end{remark}

\subsection{General case of Brownian representation} In the preceding subsection  it 
was shown that a quite general class of cylindrical martingales with absolutely 
continuous covariation can be represented as stochastic integrals with respect to 
$W_H$, and we proved some results on stochastic integrability for them (Proposition 
\ref{stochintabscont} and Theorem \ref{theorem:inttheory}). 
Unfortunately, such results do not hold in the general case,
and it will be shown in Example \ref{ex:maincounterex} that $G$ from \eqref{theorembrownianrepresenrationeq1}
does not always exists.
The construction in this example uses two simple remarks on linear operators. 
For linear spaces $X$ and $Y$ we denote the linear space of all linear operators from $X$ to $Y$ by $L(X,Y)$,
 $L(X):=L(X,X)$.

\begin{remark}\label{rem:Hamel}
 Let $(x_{\alpha})_{\alpha \in \Lambda}\subset X$ be a Hamel (or algebraic) 
basis of $X$ (see \cite[Problem~13.4]{KF} or \cite{Hamel}). Then one can 
uniquely determine $A \in L(X,Y)$ only by its values on $(x_{\alpha})_{\alpha 
\in \Lambda}$. Indeed, for each $x$ there exist unique $N \geq 0$, 
$\alpha_1,\ldots,\alpha_N \in \Lambda$, and $c_1,\ldots,c_N\in \mathbb R$ such 
that $x = \sum_{n=1}^N c_nx_{\alpha_n}$, so one can define $Ax$ as follows: $Ax 
:=\sum_{n=1}^N c_n A x_{\alpha_n}$.
\end{remark}

\begin{remark}\label{rem:extention}
Any linear functional $\ell:X_0\to \mathbb R$ defined on a linear subspace 
$X_0$ of a Banach space $X$ can be extended linearly to $X$ using the fact that $X$ 
has a Hamel 
basis (see \cite[Part I.11]{TL}, also \cite{ArYar}). The same holds for 
operators: if $A \in  L(X_0,Y)$, where $X_0$ is a linear subspace of $X$, one 
can extend $A$ to a linear operator from $X$ to $Y$ using the Hamel basis of 
$X$. Surely this extension is not unique if $X_0\subsetneqq X$.
\end{remark}

\begin{example}\label{ex:maincounterex}
We will show that for a Hilbert space $H$ there exists $M\in \mathcal M_{\rm cyl}^{\rm loc}(H)$ which is not closed operator-Brownian representable.
 Let $(\Omega, \mathcal F, \mathbb P)$ be a probability space with a filtration 
$\mathbb F = (\mathcal F_t)_{t\geq 0}$ satisfying the usual conditions, $W: 
\mathbb R_+ \times\Omega\to \mathbb R$ be a Brownian motion. Without loss of 
generality suppose that $\mathbb F$ is generated by $W$. Let $H$ be a separable 
Hilbert space with an orthonormal basis $(h_n)_{n \geq 1}$.
Let $(\xi_n)_{n \geq 1}$ be the following sequence of random variables:
for each fixed $n \geq 1$ $\xi_n \in L^2(\Omega)$ is a measurable integer-valued 
function of
$W(2^{-n+1}) -W(2^{-n}) $ such that
$$
\mathbb P (\xi_n = k) = 2^{-n+1}\mathbf 1_{2^{n-1}\leq k < 2^n},\;\;\; k\in 
\mathbb N.
$$
It is easy to see that $(\xi_n)_{n \geq 1}$ are mutually independent
and each $\xi_n$ is $\mathcal F_{2^{-n+1}}$-measurable. For all $n\geq 1$ set 
$c_n = 2^{\frac n4}$. Consider linear functional-valued function $\ell : \mathbb 
R_+\times\Omega \to L(H,\mathbb R)$ defined as
\begin{equation}\label{strangelinfunc}
 \ell (h) = \sum_{n=1}^{\infty} c_n\langle h,h_{\xi_n}\rangle
\end{equation}
for all $h\in H$ such that $\sum_{n=0}^{\infty}|c_n\langle h,h_{\xi_n}\rangle|$ 
converges, and extended linearly to the whole $H$ thanks to Remark 
\ref{rem:extention}. Fix $h \in H$. Let $\tilde h = \sum_{n=1}^{\infty} |\langle 
h,h_n\rangle|h_n$ and $a = \sum_{n=0}^{\infty}\sum_{k=2^{n-1}}^{2^n-1}2^{-\frac 
{3n}{4}}h_k\in H$. Then
$$
\mathbb E \sum_{n=0}^{N}|c_n\langle h,h_{\xi_n}\rangle| =\sum_{n=0}^{N} 
\sum_{k=2^{n-1}}^{2^n-1} \frac {c_n|\langle 
h,h_k\rangle|}{2^{n-1}}=\sum_{n=0}^{N} \sum_{k=2^{n-1}}^{2^n-1} \frac{2^{-\frac 
{3n}{4}}|\langle h,h_k\rangle|}{2}\leq \frac 12\langle\tilde h,a\rangle.
$$
Hence by the dominated convergence theorem $\lim_{N\to\infty} 
\sum_{n=0}^{N}|c_n\langle h,h_{\xi_n}\rangle|$ exists a.s., so, for each fixed 
$h\in H$ formula \eqref{strangelinfunc} holds a.s. Consider the stochastic integral
$$
M_t(h) = \int_0^{t}\mathbf 1_{[1,2]}(s)\ell(h)dW_s,
$$
since integrand is predictable. Moreover, due to the mutual independence of 
$(\xi_n)_{n\geq 1}$
\begin{align*}
 \mathbb E\int_{\mathbb R_+}\mathbf 1_{[1,2]}(s)(\ell (h))^2ds &= \mathbb E(\ell 
(h))^2 \leq \mathbb E(\ell (\tilde h))^2
 =\mathbb E\Bigl( \sum_{n=1}^{\infty} c_n |\langle h,h_{\xi_n}\rangle| 
\Bigr)^2\\
 &=\mathbb E \sum_{n=1}^{\infty}\sum_{m=1}^{\infty} c_nc_m |\langle 
h,h_{\xi_n}\rangle||\langle h,h_{\xi_m}\rangle| \\
 &= \sum_{n=1}^{\infty} c_n^2\sum_{k=2^{n-1}}^{2^n-1} \frac{|\langle 
h,h_{k}\rangle|^2}{2^{n-1}} \\
 &+2\sum_{n\neq m}c_n c_m 
\sum_{k=2^{n-1}}^{2^n-1}\sum_{l=2^{m-1}}^{2^m-1}\frac{|\langle 
h,h_{k}\rangle||\langle h,h_{l}\rangle|}{2^{n-1}2^{m-1}}\\
 &\leq \|h\|^2 + \sum_{n,m = 
1}^{\infty}\sum_{k=2^{n-1}}^{2^n-1}\sum_{l=2^{m-1}}^{2^m-1}\frac{|\langle 
h,h_{k}\rangle||\langle h,h_{l}\rangle|}{2^{n-1}2^{m-1}}\\
 &=\|h\|^2 + \Bigl(\sum_{n=1}^{\infty}\sum_{k=2^{n-1}}^{2^n-1} \frac{c_n|\langle 
h,h_{k}\rangle|}{2^{n-1}}\Bigr)^2\\ 
 &= \|h\|^2 + \frac 14 \langle\tilde h,a\rangle^2\leq (1+\frac 14 \|a\|^2) 
\|h\|^2,
\end{align*}
so thanks to \cite[Lemma~17.10]{Kal} $M_t(h)$ is an $L^2$-martingale. But the above
computations also show that by \cite[Proposition~17.6]{Kal} $M(h)\to 0$ in 
the ucp topology as $h \to 0$, which means that $M$ is a cylindrical martingale 
as a continuous linear mapping from $H$ to $\mathcal M^{\rm loc}$.

Now our aim is to prove that $M$ is not closed operator-Brownian representable. 
Suppose that there exist an $H$-cylindrical Brownian motion $W_H:\mathbb R_+ 
\times \overline {\Omega} \to \mathbb R$ on an enlarged probability space 
$(\overline {\Omega}, \overline{\mathcal F},\overline {\mathbb P})$ with an 
enlarged filtration $\overline{\mathbb F} = (\overline{\mathcal F}_t)_{t\geq 0}$ 
(we may use the same Hilbert space $H$ since all 
separable infinite-dimensional Hilbert spaces are isometrically isomorphic) and 
a closed operator-valued $H$-strongly measurable function $G:\mathbb R_+\times 
\overline{\Omega} \to \mathcal L_{cl}(H)$ such that
$M_t(h) = \int_0^t (G^*h)^*dW_H$.

Since $M|_{[0,1]} = 0$ we can assume that $G^*|_{[0,1]} = 0$. 
Because of the structure of $M$, for each pair of vectors $h,g\in H$ there exist 
$\overline{\mathcal F}_1$-measurable $a,b\in L^0(\Omega)$ such that $a=\ell(g)$ and $b=-\ell(h)$ 
and $aM_t(h)+bM_t(g) = 0$ a.s. for all $t\geq 0$. As $a,b$ are $\overline{\mathcal 
F}_1$-measurable, then for $t\geq 1$ one can put $a,b$ under the integral: 
\begin{align*}
 0=aM_t(h)+bM_t(g) &= a\int_1^t (G^*h)^* dW_H + b\int_1^t (G^*g)^* dW_H \\
 &= \int_1^t (aG^*h+bG^*g)^* dW_H = \int_0^t (aG^*h+bG^*g)^* dW_H,
\end{align*}
and by the It\^o isometry \cite[Proposition~4.20]{DPZ} $aG^*h+bG^*g = 0$ 
$\overline{\mathbb P}\otimes \ud s$-a.s.\ This means 
that $G^*h$ 
and $G^*g$ are collinear $\overline{\mathbb P}\otimes \ud s$-a.s.\ if $a$ and $b$ are nonzero a.s.\ If for instance $a=\ell(g) = 0$ on a set of positive measure $A \in \overline{\mathcal F}_1$, then $M(g)\mathbf 1_{A} = 0$, and consequently $G^*g \mathbf 1_{A}= 0$ $\overline{\mathbb P}\otimes \ud s$-a.s., hence $G^*h$ and $G^* g$ are collinear $\overline{\mathbb P}\otimes \ud s$-a.s.

Taking an 
orthonormal basis $(h_i)_{i\geq 1}$ of $H$ 
it follows that $\overline{\mathbb P}\otimes \ud s$-a.s.
$G^*h_i$ and 
$G^*h_j$ are collinear for all $i,j$, and by the 
closability of $G^*$ one has that $\text{ran } (G^*)$ consists of one vector 
$\overline{\mathbb P}\otimes\ud s$-a.s.\ But this means that $G^*$ is a 
projection on a one-dimensional subspace, so there exist $h'_G,h''_G:\mathbb R_+ 
\times \overline{\Omega} \to H$ such that $G^*h = \langle h,h'_G \rangle h''_G$.

Since the derivative of an absolutely continuous function is defined uniquely, 
$\overline{\mathbb P}$-a.s.\ for a.e. $t\in[1,2]$
$$
[Mh_i]'_{t}=\ell(h_i)^2 = \|G^*(t)h_i\|^2 = \|h''_G(t)\|^2 \langle 
h_i,h'_G(t)\rangle^2.
$$ 
But the series of positive functions $\sum_{i=1}^{\infty} \|h''_G\|^2 \langle 
h_i,h'_G\rangle^2 = \|h'_G\|^2\|h''_G\|^2$ converges $\overline{\mathbb 
P}\otimes \ud s$-a.s., which does not hold true for 
$\sum_{i=1}^{\infty}\ell(h_i)^2$ (because linear functional $\ell$ is unbounded 
$\mathbb P\times \ud s$-a.s.\ thanks to the choice of $\{c_n\}_{n\geq 1}$).
\end{example}

\begin{remark}
 Using the previous example and \cite[Example~3.22]{VY} one can see that in 
general for a separable Hilbert space $H$ the following proper inclusions hold: 
 $$
 \mathcal M_{\rm a.c.v.}^{\rm loc}(H) \subsetneqq \mathcal M_{\rm bdd}^{\rm 
cyl}(H)\subsetneqq \mathcal M_{\rm cl}^{\rm cyl}(H)\subsetneqq \mathcal M_{\rm 
a.c.c.}^{\rm cyl}(H),
 $$
 where $\mathcal M_{\rm a.c.v.}^{\rm loc}(H)$ is the subspace of $ \mathcal M_{\rm 
loc}^{\rm cyl}(X)$ with an absolutely continuous quadratic variation (the
quadratic variation of a cylindrical continuous local martingale was defined in 
\cite{VY}), and $\mathcal M_{\rm bdd}^{\rm cyl}(H)$ is the linear space of 
cylindrical continuous local martingales with a bounded operator-generated 
covariation (considered for instance in \cite{Ond1,Ond2}).
\end{remark}

In the general case one can still represent a cylindrical martingale with an 
absolutely continuous covariation as a stochastic integral with respect to $W_H$, but then 
one has to use linear operator-valued functions instead of $\mathcal 
L_{cl}(H,X)$-valued, and some important properties are lost.

\begin{theorem}\label{theorem:mainrepresentation}
 Let $X$ be a Banach space with separable dual, and let $M \in \mathcal M_{\rm cyl}^{\rm 
loc}(X)$. Then $M$ is Brownian representable
 if and only if it is with an absolutely continuous covariation.
\end{theorem}

We will need the following lemma.

\begin{lemma}\label{lemma:abscontquad}
 Let $H$ be a separable Hilbert space, and let $M:\mathbb R_+ \times \Omega \to H$ be a 
continuous local martingale such that $\langle M,h\rangle$ has an absolutely 
continuous variation for all $h \in H$ a.s. Then $[M]$ also has an absolutely 
continuous version. Moreover, there exists scalarly measurable positive 
Hilbert-Schmidt operator-valued $\Phi:\mathbb R_+\times \Omega \to \mathcal 
L(H)$ such that a.s.\ $[\langle M,h\rangle,\langle M,g\rangle] = 
\int_0^{\cdot}\langle \Phi h,\Phi g\rangle ds$.
\end{lemma}

\begin{proof}\label{[M]isabscont}
Since $H$ is a separable Hilbert space one sees that
\begin{equation}\label{[M]isabscont1}
 [M]_t = \sum_{n=1}^{\infty}[\langle M,e_n\rangle]_t\;\;\; a.s.\; \forall t\geq 
0
\end{equation}
for any given orthonormal basis $(e_n)_{n\geq 1}$ of $H$. Let $f_n:\mathbb 
R_+\times\Omega \to \mathbb R_+$ be such that $[\langle M,e_n\rangle]_t = 
\int_0^t f_n(s) ds$ a.s. $\forall n\geq 1, t\geq 0$. Then thanks to 
\eqref{[M]isabscont1} and the dominated convergence theorem $\sum_{n=1}^{\infty} 
f_n$ converges in $L^1_{\rm loc}(\mathbb R_+)$ a.e. Let $f := 
\sum_{n=1}^{\infty} f_n$. Then $[M]_t = \int_0^t f(s)ds$ a.s. for all $t\geq 0$, 
which means that $[M]_t$ is absolutely continuous a.s.

The second part is an easy consequence of \cite[Theorem~14.3(2)]{MP}.
\end{proof}

\begin{proof}[Proof of Theorem \ref{theorem:mainrepresentation}]
 One direction is obvious. Now let $M$ be with an absolutely continuous 
covariation. First suppose that $X$ is a Hilbert space. Consider a 
Hilbert-Schmidt operator $A$ with zero kernel and dense range. For instance set 
$Ah_n = \frac{1}{n}h_n$ for some orthonormal basis $(h_n)_{n\geq 1}$ of $X$. 
Then according to \cite[Theorem~A]{JKFR} $M(A(\cdot))$ admits a local martingale 
version, namely there exists a continuous local martingale 
$\widetilde{M}:\mathbb R_+\times \Omega \to X$ such that
 $M(Ax)$ and $\langle \widetilde M,x\rangle$ are indistinguishable for each $x 
\in X$. Notice that $\widetilde M$ has an absolutely continuous quadratic 
variation by Lemma \ref{lemma:abscontquad}. Also by \cite[Theorem~8.2]{DPZ} 
there exists an enlarged probability space $(\overline{\Omega}, 
\overline{\mathcal F}, \overline{\mathbb P})$ with an enlarged filtration 
$\overline{\mathbb F} = (\overline{\mathcal F}_t)_{t\geq 0}$ such that there 
exist an $X$-cylindrical Brownian motion $W_X:\mathbb R_+\times\overline{\Omega} 
\times X \to \mathbb R$ and $X$-strongly progressively measurable $\Phi:\mathbb 
R_+\times \overline{\Omega} \to \mathcal L(X)$ such that $\widetilde M = 
\int_0^{\cdot} \Phi^* dW_X$.

 Now fix $h \notin \text{ran } (A)$. Let $(x_n)_{n\geq 1}$ be a $\mathbb 
Q-\text{span}$ of $(h_n)_{n\geq1}$. Denote by $f_h, f_{nh}$ the derivatives of 
$[Mh]$ and $[Mh,W_Xx_n]$ in time respectively. For each $n\geq 1$ and for each 
segment $I\subset\mathbb R_+$ $|\mu_{[Mh,W_Xx_n]}|(I)\leq 
\mu_{[Mh]}^{1/2}(I)\mu_{[W_Xx_n]}^{1/2}(I)$ a.s.\ by \cite[Proposition 
17.9]{Kal} So, according to \cite[Theorem~5.8.8]{Bog}, $|f_{nh}(t)|\leq 
(f_h(t))^{\frac 12}\|x_n\|$ a.s. for almost all $t\geq 0$. One can modify 
$f_{nh}$ on $(x_n)_{n\geq 1}$ in a linear way, so it defines a bounded linear 
functional on $\text{span }(h_n)_{n\geq 1}$. Therefore there exists scalarly 
progressively measurable $a_h:\mathbb R_+\times\Omega \to X$ such that $f_{nh} = 
\langle a_h,x_n\rangle$ a.s. for almost all $t\geq 0$.

 Now consider $N = \int_0^{\cdot} a_h^*dW_X$. Then $[Mh,W_Xg]=[N,W_Xg]$ for all 
$g\in H$, and consequently $[Mh, \Phi\cdot W_H]=[N,\Phi\cdot W_H]$ for each 
$\overline{\mathbb F}$-progressively measurable stochastically integrable 
$\Phi:\mathbb R_+\times \Omega \to H$, and so $[Mh, Mg]=[N,Mg]$ for all $g\in 
\text{ran }(A)$ and $[Mh,N]=[N]$. Let $(g_n)_{n\geq 1}\subset \text{ran } (A)$ be 
such that $g_n\to h$. Then $[Mh-Mg_n]\to 0$ in ucp, and so
 \begin{align*}
 [N-Mg_n] &= [N]+[Mg_n]-2[N,Mg_n] \\
 &=([N]-[N,Mg_n])+([Mg_n]-[Mh,Mg_n])\\
 &\to ([N]-[N,Mh]) + ([Mh]-[Mh]) = 0
 \end{align*}
 in ucp, and therefore $N$ and $Mh$ are indistinguishable.

 For a general Banach space $X$, define a Hilbert space $H$ and 
a dense embedding $j:X\hookrightarrow H$ as in 
\cite[p.154]{Kuo}. Let $(h_{\alpha})_{\alpha \in \Lambda}$ be a Hamel basis of 
$H$ and $(x^*_{\beta})_{\beta \in \Delta}\cup (h_{\alpha})_{\alpha \in \Lambda}$ 
be a Hamel basis of $X^*$ (thanks to the embedding $H\hookrightarrow X^*$ and 
\cite[Theorem 1.4.5]{Ed}). Let $W_H$ be a Brownian motion constructed as above 
for $M|_{H}\in \mathcal M^{\rm cyl}_{\rm loc}(H)$, i.e. for each $\alpha \in 
\Lambda$ there exists progressively measurable $a_{h_{\alpha}}:\mathbb R_+ 
\times \Omega \to H$ such that $Mh_{\alpha}$ and $\int_0^{\cdot}a_{h_{\alpha}}^* 
\ud W_H$ are indistinguishable. Using the same technique and the fact that 
$j^*:H\hookrightarrow X^*$ is a dense embedding, for each $\beta \in \Delta$ one 
can define progressively measurable $a_{x^*_{\beta}}:\mathbb R_+ \times \Omega 
\to H$ such that $Mx^*_{\beta}$ and $\int_0^{\cdot}a_{x^*_{\beta}}^* \ud W_H$ 
are indistinguishable. Using to Remark \ref{rem:Hamel}, we can now define an $X^*$-
strongly progressively measurable operator-valued function $F:\mathbb R_+ \times 
\Omega \to L(X^*,H)$ such that for each $x^* \in X^*$ a.s.\ $Mx^* = 
\int_0^{\cdot}(Fx^*)^*\ud W_H$.
\end{proof}

The next theorem is an obvious corollary of \cite[Theorem~3.6]{NVW} and 
\cite[Theorem~5.13]{NVW}.

\begin{theorem}
 Let $X$ be a UMD Banach space, $H$ be a Hilbert space, $W_H:\mathbb R_+\times 
\Omega\times H \to \mathbb R$ be an $H$-cylindrical Brownian motion, $M\in 
\mathcal M_{\rm cyl}^{\rm loc}(X)$, and $\Phi:\mathbb R_+\times \Omega \to L(X^*,H)$ 
be scalarly predictable measurable with respect to filtration $\mathbb F_{W_H}$ 
generated by $W_H$ such that $M = \int_0^{\cdot} \Phi \ud W_H$. Then there 
exists an $X$-valued continuous local martingale $\widetilde M:\mathbb 
R_+\times\Omega \to X$ such that $Mx^* = \langle\widetilde M,x^*\rangle$ for 
each $x^* \in X^*$ if and only if $\Phi\in \mathcal L(X^*,H)$ $\mathbb P \times 
\ud s$-a.s.\ and $\Phi^* \in \gamma(L^2(\mathbb R_+;H),X)$ a.s.
\end{theorem}

\subsection{Time-change} A family $\tau = (\tau_s)_{s\geq 0}$ of finite stopping 
times is called a {\it finite random time change} if it is non-decreasing and 
right-continuous. If $\mathbb F$ is right-continuous, then
by to \cite[Lemma~7.3]{Kal} the {\it induced 
filtration} $\mathbb G = (\mathcal G_s)_{s \geq 0} = (\mathcal 
F_{\tau_s})_{s\geq 0}$ (see \cite[Chapter~7]{Kal}) is right-continuous as well.
$M\in\mathcal M_{\rm cyl}^{\rm loc}(X)$ is said to be {\it $\tau$-continuous} if 
a.s.\ for each $x^*\in X^*$, $Mx^*$ (and, thanks to \cite[Problem 17.3]{Kal} 
equivalently $[Mx^*]$) is a constant on every interval $[\tau_{s-}, \tau_s]$, $s 
\geq 0$, where we set $\tau_{0-} = 0$.

\begin{remark}\label{rem:timechangecyl}
 Note that if $M\in \mathcal M_{\rm cyl}^{\rm loc}(X)$ is $\tau$-continuous for 
a given time-change~$\tau$, then $M\circ \tau \in \mathcal M_{\rm cyl}^{\rm 
loc}(X)$. Indeed, for each given $x^* \in X^*$ one concludes thanks to 
\cite[Proposition 17.24]{Kal} $Mx^* \circ \tau$ is a continuous local 
martingale. Also for a given vanishing sequence $(x_n^*)_{n\geq 1}\subset X^*$ 
one can easily prove that $Mx_n^* \circ \tau \to 0$ in the ucp topology by using 
the stopping time argument and the fact that $Mx_n^* \to 0$ in the ucp topology 
by the definition of $\mathcal M_{\rm cyl}^{\rm loc}(X)$.
\end{remark}

The following natural question arise: does there exist a suitable time-change making a 
given $M\in \mathcal M_{\rm cyl}^{\rm loc}(X)$ Brownian representable? The answer is given in the following theorem.

\begin{theorem}\label{thm:existstimechangforabscont}
Let $X$ be a Banach space with a separable dual space, $M \in \mathcal M_{\rm 
cyl}^{\rm loc}(X)$. Then there exists a time-change $(\tau_s)_{s\geq 0}$ such 
that $M\circ \tau$ is with an absolutely continuous covariation, i.e. Brownian 
representable.
\end{theorem}

\begin{proof}
Let $(x_n^*)_{n\geq 1}\subset X^*$ be a dense subset of the unit ball of $X^*$. 
Consider the  increasing predictable 
process $F:\mathbb R_+ \times \Omega \to \mathbb R_+$ given by $F(t) = 
t(1+\sum_{n=1}^{\infty} \frac {1}{2^n}\arctan ([Mx_n^*]_t))$, and consider the time-change
 $\tau_s = \inf\{t\geq 0:F_t>s\}$ for $s\geq 0$. This time-change is 
finite since $\lim_{t\to \infty}F(t) = \infty$.

For each fixed $\omega\in \Omega$ and $n\geq 1$ one has $\mu_{[Mx_n^*]}\ll 
\mu_F$. Then by Lemma \ref{lemma:appB1} one sees that $\mu_{[Mx^*]}\ll \mu_F$ 
a.s.\ for each $x^*\in X^*$, so $M$ is $\tau$-continuous. Moreover, for each 
$x^* \in X^*$ a.s. one has $\mu_{[Mx^*\circ \tau]} =\mu_{[Mx^*]\circ \tau}\ll 
\mu_{F\circ \tau}$, where the last measure is a Lebesgue measure on $\mathbb 
R_+$, so by Remark \ref{rem:timechangecyl} $M\circ \tau \in\mathcal M_{\rm 
cyl}^{\rm loc}(X)$ is with an absolutely continuous covariation.
\end{proof}

Let $X$ be a separable Banach space, and let $\widetilde{M}\in \mathcal M^{\rm loc}(X)$. 
Then $\widetilde{M}$ is {\it weakly Brownian representable} if there exist a 
Hilbert space $H$, an $H$-cylindrical Brownian motion $W_H$ and a function $G:\mathbb 
R_+\times\Omega\to L(X^*,H)$ such that for each $x^* \in X^*$ the function 
$Gx^*$ is stochastically integrable w.r.t. $W_H$ and $\langle \widetilde 
M,x^*\rangle = \int_0^{\cdot} Gx^*\ud W_H$ a.s.\ Thanks to \cite[Part 3.3]{VY} 
there exists an associated cylindrical continuous local martingale $M \in 
\mathcal M_{\rm cyl}^{\rm loc}(X)$, so the following corollary of Theorem 
\ref{thm:existstimechangforabscont} holds.

\begin{corollary}
Let $X$ be a Banach space with a separable dual space, and let $\widetilde {M}:\mathbb 
R_+\times \Omega \to X$ be a continuous local martingale. Then there exists a 
time-change $(\tau_s)_{s\geq 0}$ such that $\widetilde M\circ \tau$ is weakly 
Brownian representable.
\end{corollary}

\begin{remark}
 Unfortunately we do not see a way to prove an analogue of Theorem 
\ref{theorem:inttheory}-\ref{thm:closedcovmain} in the present general case even for 
an $X^*$-valued integrand. The main difficulty is the discontinuity of the 
corresponding operator-valued function. One of course can prove such an analogue 
for integrands with values in a given finite-dimensional subspace of $X^*$, 
but this would amount to a stochastic integral with respect to an $\mathbb 
R^d$-valued continuous local martingale for some $d\geq 1$; the theory for this 
has been developed by classical works such as e.g. \cite{MP}.
\end{remark}

\appendix

\section{Technical lemmas on measurable closed operator-valued functions}\

The following lemma shows that a Borel bounded function of a closed 
operator-valued scalarly measurable function is again an operator-valued 
scalarly measurable function.

\begin{lemma}\label{Borelcalculus1}
 Let $(S,\Sigma)$ be a measurable space, $H$ be a separable Hilbert space, and  $f:S 
\to \mathcal L_{cl}(H)$
 be such that $(h_i)_{i=1}^{\infty}\subset \text{dom }(f^*(s))$ for each $s\in S$ 
and $(f^*h_i)_{i\geq 1}$ are measurable for some fixed orthonormal basis
 $(h_i)_{i\geq 1}$ of $H$. Let $g:\mathbb R \to \mathbb R$ be finite Borel 
measurable.
 Then $g(f^*f):S \to \mathcal L(H)$ is well-defined and scalarly measurable,
 $\|g(f^*f)(s)\|\leq \|g\|_{L^{\infty}(\mathbb R)}$ for each $s \in S$.
\end{lemma}

To prove this lemma we will need two more lemmas.

\begin{lemma}\label{f^*fselfadjoint}
 Let $H$ be a Hilbert space, and let $T\in \mathcal L_{cl}(H)$. Then $T^*T\in \mathcal 
L_{cl}(H)$.
\end{lemma}
\begin{proof}
 According to \cite[Chapter~118]{RSN} there exists a bounded positive operator 
$B\in\mathcal L(H)$ such that
 $B = (1+T^*T)^{-1}$ and $\text{ran }(B) = \text{dom } (T^*T)$. Since $\ker B = 
\{0\}$ by the construction,
 $T^*T$ is densely defined. Furthermore since $B$ is closed, by 
\cite[Proposition~II.6.3]{Yos} $T^*T = B^{-1}-1$ is also closed.
\end{proof}

\begin{lemma}\label{unboundedoperatorlimit}
 Let $H$ be a Hilbert space, $A \subset \mathcal L_{cl}(H)$ be such that 
$(h_n)_{n=1}^{\infty}\subset\text{dom }(A^*)$ for a~certain orthonormal basis 
$(h_n)_{n=1}^{\infty}$ of $H$. For each $n\geq 1$ let $P_n\in\mathcal L(H)$ be 
the orthogonal projection onto $\text{span}(h_1,\ldots,h_n)$, and set $A_n := P_nA$.
Then
\begin{itemize}
\item[(i)] the operators
$((i+A_n^*A_n)^{-1})_{n\geq 1}$, $(i+A^*A)^{-1}$ are bounded;
 \item[(ii)]$(i+A_n^*A_n)^{-1}h \to (i+A^*A)^{-1}h$ weakly for each $h\in H$.
\end{itemize}
\end{lemma}
Using \cite[Problem~III.5.26]{Kato1} we note that $A_n \subset 
(A^*P_n)^*\in\mathcal L_{cl}(H)$ for all $n$.

\begin{proof}
 The first part is an easy consequence of \cite[Theorem~XI.8.1]{Yos}. To prove 
the second part we use the formula
$$
((i+A_n^*A_n)^{-1}-(i+A^*A)^{-1})h = (i + A^*A)^{-1}(A^*A-A_n^*A_n)(i + 
A_n^*A_n)^{-1}h,\;\;\; h\in H,
$$
which follows from the fact that for each $n\geq 1$ there exists 
$\tilde{h}\in H$ such that $h=(i+A_n^*A_n)\tilde h$. Thanks to \cite[p.347]{RSN} 
$\text{ran }(A^*A-i)^{-1} \subset \text{dom } (A^*A)$, and therefore for each $h,g 
\in H$ and $n\geq 0$
\begin{align*}
 \langle((i+A_n^*A_n)^{-1}-(i+A^*A)^{-1})h,g\rangle &= \langle (i + 
A^*A)^{-1}(A^*A-A_n^*A_n)(i + A_n^*A_n)^{-1}h,g\rangle\\
 &= \langle (A^*A-A_n^*A_n)(i + A_n^*A_n)^{-1}h,(A^*A-i)^{-1}g\rangle\\
  &= \langle A(i + A_n^*A_n)^{-1}h,A(A^*A-i)^{-1}g\rangle\\
  &- \langle A_n(i + A_n^*A_n)^{-1}h,A_n(A^*A-i)^{-1}g\rangle\\
  &= \langle (A-A_n)(i + A_n^*A_n)^{-1}h,A(A^*A-i)^{-1}g\rangle\\
  &- \langle A_n(i + A_n^*A_n)^{-1}h,(A_n-A)(A^*A-i)^{-1}g\rangle\\
  &=\langle (i + A_n^*A_n)^{-1}h,(I-P_n)A^*A(A^*A-i)^{-1}g\rangle\\
  &- \langle A_n(i + A_n^*A_n)^{-1}h,(A_n-A)(A^*A-i)^{-1}g\rangle.
\end{align*}
Let $H_n = \text{span }(h_1,\ldots,h_n)$. Note that $(A^*A-i)^{-1}g \in 
\text{dom } (A^*A) \subset \text{dom } (A)$ (see \cite[p.347]{RSN}) and 
$\text{ran }(A-A_n)\perp H_n$, so $(A_n-A)(A^*A-i)^{-1}g \in \text{ran } 
(A_n-A)\perp H_n$. Also $A_n(i + A_n^*A_n)^{-1}h \in \text{ran } (A_n) \subset 
H_n$. Therefore, for each $n\geq 1$
$$
\langle A_n(i + A_n^*A_n)^{-1}h,(A_n-A)(A^*A-i)^{-1}g\rangle = 0,
$$
and for the sequence $(Q_n)_{n\geq 1} := (I-P_n)_{n\geq 1}\subset \mathcal L(H)$ that vanish weakly
\begin{align*}
 \langle((i+A_n^*A_n)^{-1}-(i+A^*A)^{-1})h,g\rangle
  &=\langle (i + A_n^*A_n)^{-1}h,Q_nA^*A(A^*A-i)^{-1}g\rangle\\
  &\leq \|(i + A_n^*A_n)^{-1}h\|\|Q_nA^*A(A^*A-i)^{-1}g\|,
\end{align*}
which vanishes as $n$ tends to infinity, where according to 
\cite[Example~VIII.1.4]{Yos} $\|(i+A_n^*A_n)^{-1}\|\leq 1$ for each $n \geq 1$.

\end{proof}

\begin{proof}[Proof of Lemma \ref{Borelcalculus1}]
 First of all, one can construct $g(f^*f)$ (without proving measurability 
property) by guiding \cite[Chapter~120]{RSN}
by to constructing a spectral family of $f^*f(s)$ for each fixed $s \in S$,
 and further using bounded calculus \cite[Chapter~126]{RSN} for the 
corresponding spectral family.

 To prove scalar measurability we have to plunge into the construction of the 
spectral family.
 Let us firstly prove that $(i+f^*f)^{-1}$ is scalarly measurable. Notice that 
by \cite[Theorem~XI.8.1]{Yos}
 $(i+f^*f(s))^{-1}\in \mathcal L(H)$ for each $s\in S$. We will proceed in two 
steps:

{\it Step 1.} Suppose that $f(s)$ is bounded for all $s \in S$.
 Fix $k\geq 1$. Consider $\text{span} ((i+f^*f) h_i)_{1\leq i\leq k}$. This is a 
$k$-dimensional subspace of $H$ for each $s \in S$ since $i+f^*f$ is invertible. 
Let $\tilde P_{k}$ be defined as an orthogonal projection
onto $\text{span} ((i+f^*f) h_i)_{1\leq i\leq k}$, $(g_i)_{1\leq i\leq k}$ be 
obtained from $((i+f^*f)h_i)_{1\leq i\leq k}$
by the Gram--Schmidt process. These vectors are orthonormal and measurable
because $(\langle (i+f^*f)h_i, (i+f^*f)h_j\rangle)_{1\leq i,j\leq k}$ are 
measurable, so $\tilde P_{k}$
is scalarly measurable. Moreover,
the transformation matrix $C=(c_{ij})_{1\leq i,j \leq k}$ such that
\begin{align*}
 g_i = \sum_{j=1}^k c_{ij}(i+f^*f)h_j, \;\;\; 1\leq i\leq k,
\end{align*}
has measurable elements and invertible since by \cite[Theorem~XI.8.1]{Yos} $\ker 
(i+f^*f) = 0$.
So, one can define the scalarly measurable inverse $(i+f^*f)^{-1}\tilde P_k$:
$$
(i+f^*f)^{-1}\tilde P_k g = \sum_{j=1}^k d_{ij}\langle g,g_j\rangle h_{j},
$$
where $D = \{d_{ij}\}_{1\leq i,j \leq k} = C^{-1}$.

Now fix $s\in S$, $g \in H$. Let $x_k = (i+f^*(s)f(s))^{-1}\tilde P_k(s) g$. 
Since $(i+f^*(s)f(s))^{-1}$ is a bounded operator
and $\lim_{k\to \infty} \tilde P_k g = g$ (because by \cite[Theorem~XI.8.1]{Yos} 
$\text{ran}(i+f^*(s)f(s)) = H$), then
$(i+f^*(s)f(s))^{-1}g = x:=\lim_{k\to\infty}x_k$. So $(i+f^*f)^{-1}g$ is 
measurable as a limit of
measurable functions.

{\it Step 2.} In general case one can consider the function $f_k = P_kf$ for 
each $k \geq 1$,
where $P_k\in \mathcal L(H)$ is an orthogonal projection onto 
$\text{span}(h_1,\ldots,h_k)$. Then by Lemma \ref{unboundedoperatorlimit} and 
thanks to the step 1 applied to $f_k$ one can prove
that $(i+f^*(s)f(s))^{-1}h$ is a weak limit of measurable functions 
$(i+f_k^*(s)f_k(s))^{-1}h$, so, since $H$ is separable, it is measurable.

\vspace{.05cm}

For the same reason $(f^*f-i)^{-1}$ is scalarly measurable, therefore
$(1+(f^*f)^2)^{-1} = (i+f^*f)^{-1}(f^*f-i)^{-1}$ is scalarly measurable.

Now guiding by the construction in \cite[Chapter~120]{RSN} one can consider the 
sequence of orthogonal
Hilbert spaces $\{H_i(s)\}_{i\geq 1}$ depending on $s$ such that orthogonal 
projection $P_{H_i}$ onto $H_i$
is scalarly measurable (thanks to \cite[Proposition~32]{Ond1} and the fact that
$P_{H_i} = \mathbf 1_{(\frac 1{i+1},\frac 1i]}((1+(f^*f)^2)^{-1})$). Then by 
\cite[Chapter~120]{RSN}
$f^*(s)f(s)P_{H_i(s)}$ is bounded for~each~$s$. Moreover, 
$\text{ran}(f^*(s)f(s)P_{H_i(s)})\subset H_i(s)$ $\forall s \in S$,
so $P_{H_i}f^*fP_{H_i} = f^*fP_{H_i}$ and $g(f^*fP_{H_i}):S \to \mathcal 
L(H,H_i)$ is well defined and thanks to \cite[Proposition~32]{Ond1}
scalarly measurable. Finally, since $\|g(f^*fP_{H_i})\|\leq \|g\|_{L^{\infty}}$ 
and $H = \oplus_i H_i$,
then one can define $g(f^*f):= \sum_{i=1}^{\infty}g(f^*fP_{H_i})$ as in 
\cite[Chapter~120]{RSN}, which is scalarly measurable
and by \cite[Chapter~120]{RSN} $\|g(f^*f)\|\leq \|g\|_{L^{\infty}}$ as well.
\end{proof}

\begin{corollary}
 Let $(S,\Sigma,\mu)$ be a measure space, $H$ be a separable Hilbert space, $f:S 
\to \mathcal L_{cl}(H)$
 be such that $f^*h$ is a.s. defined and measurable for each $h\in H$.
 Let $g:\mathbb R \to \mathbb R$ be finite Borel measurable.
 Then $g(f^*f):S \to \mathcal L(H)$ is well-defined and scalarly measurable,
 $\|g(f^*f)(s)\|\leq \|g\|_{L^{\infty}(\mathbb R)}$ for almost all $s \in S$.
\end{corollary}

The following lemma can be proved in the same way as the second part of \cite[Lemma 
A.1]{VY}.

 \begin{lemma}\label{technical}
Let $(S,\Sigma,\mu)$ be a measure space, $H$ be a separable Hilbert space, and 
let $X$ be a Banach space. Let $X_0\subseteq X$ be a finite dimensional 
subspace. Let $F :S \to \mathcal L_{cl}(X,H)$ be a function such that
$Fx$ is defined a.s.\ and strongly measurable for each $x\in X$.
For each $s\in S$, let $\tilde P(s)\in \calL(H)$ be the orthogonal projection 
onto $F(s) X_0$.
Then $\tilde P$ is strongly measurable. Moreover, there exists a strongly 
measurable function $L :S \to \calL(H,X)$ with values in $X_0$ such that
$FL = \tilde P$.
\end{lemma}

\section{Lemmas on absolute continuity of quadratic variations}\

The main result of this subsection provides a surprising property of a limit in the ucp topology.

\begin{lemma}\label{lemma:appB1}
 Let $(M_n)_{n\geq 1}$, $M$ be real-valued continuous local martingales, and 
$F:\mathbb R_+ \times \Omega \to \mathbb R$ be a continuous progressively 
measurable nondecreasing process such that $\mu_{[M_n]}\ll \mu_F$ a.s.\ for each 
$n\geq 0$. Let $M_n\to M$ in ucp. Then $\mu_{[M]}\ll \mu_F$ a.s.\
\end{lemma}

We will need the following lemma.

\begin{lemma}\label{lemma:appB2}
 Let $(M_n)_{n\geq 1}$ be a sequence of real-valued continuous local martingales 
starting in $0$ such that $[M_n]$ is a.s.\ absolutely continuous for each $n\geq 0$. 
Then there exists a Hilbert space $H$, an $H$-cylindrical Brownian motion $W_H$ 
on an enlarged probability space $(\overline{\Omega}, \overline{\mathcal 
F},\overline{\mathbb P})$ and a sequence of functions $(F_n)_{n\geq 1}$, 
$F_n:\mathbb R_+ \times \overline{\Omega}\to H$, $n\geq 1$ such that
 \begin{equation*}
  M_n = F_n\cdot W_H,\;\;\;n\geq 1.
 \end{equation*}
 \end{lemma}
 \begin{proof}
For each $k\geq 0$ we consider separately 
$(M_n(t)-M_n(k))_{n\geq 1}$, $k\leq t<k+1$. (The resulting cylindrical Brownian 
motions $W_H^k$ can be glued together thanks to the independence of Brownian motion 
increments).
  
  It is enough to consider the case $k=0$. For each $n\geq 1$ one can find a 
nonzero real number $a_n$ such that 
\begin{align*}
   \mathbb P\{\sup_{0\leq t\leq 1} |a_nM_n(t)|&>\frac 1{2^n}\}<\frac 1{2^n},\\
   \mathbb P\{[a_nM_n]_1&>\frac 1{2^n}\}<\frac 1{2^n}.
\end{align*}
Without loss of generality redefine $M_n :=a_nM_n$. Let $H$ be a separable 
Hilbert space with an orthonormal basis $(h_n)_{n\geq 1}$. Then 
$\sum_{n=1}^{\infty}M_nh_n$ converges uniformly a.s.\ Therefore, $\sum_{n=1}^{\infty}M_nh_n$ converges in ucp topology, and $M := \sum_{n=1}^{\infty}M_nh_n:[0,1]\times 
\Omega \to H$ is an $H$-valued continuous local martingale. Moreover, thanks to 
Lemma \ref{lemma:abscontquad} $[M]=\sum_{n=1}^{\infty}[M_n]$ a.s.\, and $[M]$ is 
absolutely continuous as a countable sum of absolutely continuous nondecreasing 
functions. Now using $H$-valued analogue of Brownian representation results one 
can find an $H$-cylindrical Brownian motion $W_H$ on an enlarged probability 
space $(\overline{\Omega}, \overline{\mathcal F},\overline{\mathbb P})$, 
operator-valued function $\Phi:\mathbb R \times \overline{\Omega}\to \mathcal 
L(H)$ such that
$$
\langle M,h\rangle = \Phi h\cdot W_H,\;\;\;h\in H.
$$
In particular
$$
M_n = \langle M,h_n\rangle = \Phi h_n\cdot W_H,\;\;\;h\in H.
$$
 \end{proof}
\begin{proof}[Proof of Lemma \ref{lemma:appB1}]
 Without loss of generality suppose that $F(0) = 0$ and $F(t)\nearrow \infty$ as 
$t\to \infty$ a.s.\ Otherwise redefine 
 $$
 F(t):= F(t)-F(0)+t,\;\;\; t\geq 0.
 $$
 Also by choosing a subsequence set $M_n$ converges to $M$ and $[M_n]$ converges 
to $[M]$ uniformly on compacts a.s.\ as $n$ goes to infinity.
 
 Let $(\tau_s)_{s\geq \infty}$ be the following time change:
 \[
 \tau_s :=
\inf\{t\geq 0:F(t)>s\}, \;\;\; s\geq 0. 
\]
 Then for each $n\geq 1$ by \cite[Proposition 17.6]{Kal} and the fact that 
$\mu_{[M_n]}\ll \mu_F$ a.s., $M_n$ is $\tau$-continuous (see \cite[Chapter 
7]{Kal}). Since $M_n$ are $\tau$-continuous and $M_n$ converge to $M$ uniformly 
on compacts a.s., then $M$ is $\tau$-continuous, and one can then define local 
martingales
 \begin{align*}
   N_n:&= M_n\circ \tau,\;\;\; n\geq 1,\\
   N :&=M\circ\tau,
 \end{align*}
 which are defined on a probability space $(\Omega, \mathcal F, \mathbb P)$ with 
an induced filtration $\mathbb G = (\mathcal G_s)_{s\geq 0} = (\mathcal 
F_{\tau_s})_{s\geq 0}$ (see \cite[Chapter 7]{Kal}). Also by \cite[Theorem 
17.24]{Kal} $[N_n] = [M_n]\circ \tau$ a.s., hence since $\mu_{[M_n]}\ll \mu_{F}$
 $$
 \mu_{[N_n]} = \mu_{[M_n]\circ \tau}\ll \mu_{F\circ \tau} = \lambda,
 $$
 so by Lemma \ref{lemma:appB2} there exists a separable Hilbert space $H$, an 
$H$-cylindrical Brownian motion $W_H$ on an enlarged probability space 
$(\overline{\Omega}, \overline{\mathcal F},\overline{\mathbb P})$ with an 
enlarged filtration $\overline{\mathbb G}$ and a sequence of functions 
$(F_n)_{n\geq 1}$, $F_n:\mathbb R_+ \times \overline{\Omega}\to H$, $n\geq 1$ 
such that
 \begin{equation*}
  N_n = F_n\cdot W_H,\;\;\;n\geq 1.
 \end{equation*}
 But we know that $N_n\to N$ uniformly on compacts a.s.\ since $\tau_s \to 
\infty$ a.s.\ as $s\to \infty$, so there exists a progressively measurable 
function $R:\mathbb R_+\times\Omega\to\mathbb R_+$, $t\mapsto 
\sup_{n\geq 1}|N_n|(t)+t$, which is nondecreasing continuous a.s. For each natural $k$ define a stopping time $\rho_k:=
\inf\{t\geq 0:R(t)>k\}$. Then $N_n^{\rho_k} \to N^{\rho_k}$ uniformly on 
compacts a.s.\ But $N_n^{\rho_k}$, $N^{\rho_k}$ are bounded by $k$, hence 
they are $L^2$-martingales, and the convergence holds in $L^2$. Hence using the 
cylindrical case of the It\^{o} isometry \cite[Remark 30]{Ond1} one can see that 
$(F_n\mathbf 1_{[0,\rho_k]})_{n\geq 1}$ converges to a function $F^k$ in 
$L^2(\mathbb R_+\times\Omega;H)$. Therefore $N^{\rho_k} = F^k\cdot W_H$, 
so $[N^{\rho_k}] = [N]^{\rho_k}=\int_0^{\cdot}\|F^k(s)\|^2\ud s$ is absolutely 
continuous. Taking $k$ to infinity and using the fact that $\rho_k \to \infty$ 
as $k\to \infty$ one can see that $[N]$ is absolutely continuous. Then
$$
\mu_{[M]}=\mu_{[N]\circ F}\ll \mu_{\lambda\circ F} =\mu_F.
$$
\end{proof}

{\bf Acknowledgments.} The author would like to thank Jan van Neerven and Mark Veraar for helpful 
comments.

\bibliographystyle{plain}

\end{document}